\documentclass[12pt]{amsart}
\usepackage{amssymb,amsmath,mathrsfs,enumerate,mathtools,comment,xcolor}
\usepackage[pdftex,bookmarks=true]{hyperref}
\usepackage{adjustbox}
\usepackage{hhline}
\usepackage{enumerate,graphicx}
\usepackage{array}
\usepackage{bbm}
\usepackage{float}
\usepackage{amsthm}
\usepackage[english]{babel} 

\theoremstyle{plain}

\newtheorem{theorem}{Theorem}[section]

\newtheorem{proposition}[theorem]{Proposition}
\newtheorem{lemma}[theorem]{Lemma}
\newtheorem{corollary}[theorem]{Corollary}
\theoremstyle{definition}

\newtheorem{definition}[theorem]{Definition}
\newtheorem{remark}[theorem]{Remark}
\numberwithin{equation}{section}

\usepackage{array}

\usepackage{bbm}
\usepackage{float}

\author[M.~Tantrawan]{Made Tantrawan}
\address{Department of Mathematics, Faculty of Mathematics and Natural Sciences, Universitas Gadjah Mada, Indonesia 55281}
\email[Corresponding author]{made.tantrawan@ugm.ac.id}

\author[D.~H.~Leung]{Denny H.~Leung}
\address{Department of Mathematics, National University of Singapore, Singapore 119076}
\email{dennyhl@u.nus.edu}

\author[N.~Gao]{Niushan Gao}
\address{Department of Mathematics, Toronto Metropolitan University, 350 Victoria Street, Toronto, Canada M5B2K3}
\email{niushan@ryerson.ca}

\title{The order-type Banach-Saks properties}

\thanks{Part of the work was done while the first author was a PhD student at NUS and  supported by NUS Research Scholarship.  The third author acknowledges the support of an NSERC grant.}


\date{\today}

\UseRawInputEncoding

\begin{document}
	
	\begin{abstract}
		The study of the Banach-Saks property in Banach spaces has a long and illustrious history.  Of late, motivated by applications in  financial mathematics, interest has arisen in the Banach-Saks type properties with respect to order convergence.  This  paper presents a study of order Banach-Saks properties in Banach function spaces, and in particular in rearrangement invariant spaces.  Among the results obtained, we provide some sufficient conditions for the (weak) order Banach-Saks property. We also characterize the (weak) order Banach-Saks property in Orlicz spaces.  It is also shown that the (weak) order Banach-Saks property is equivalent to its hereditary version.
	\end{abstract}
	
	\maketitle

\section{Introduction}

For a sequence $\{f_n\}$, we call $\left\{\frac{1}{m}\sum^m_{n=1}f_{n}\right\}$ the Ces\`{a}ro sequence of $\{f_n\}$. The classical {(resp. weak)} Banach-Saks property,  {abbreviated as $BSP$ (resp. $w$-$BSP$)}, states that every  norm bounded {(resp. weakly convergent)} sequence in a Banach space $X$ has a subsequence  whose Ces\`{a}ro sequence norm converges in $X$. Up to date, many variants of these Banach-Saks properties have been extensively studied in the literature; see, e.g., \cite{AKS,ASS,Bea,DSS,FT,GTX,LRT}. In the setting of Banach function spaces, of particular interest are the variants that  naturally suit the order structures,  namely, order convergence (also known as dominated a.e.-convergence) and unbounded order convergence (also known as a.e.-convergence). For example, in their solution to a problem regarding dual representation of risk measures on certain Orlicz spaces, Delbaen and Owari \cite{DO} proved a Banach-Saks condition with respect to dominated convergence. It has been known in which Banach function spaces $X$  every norm bounded sequence admits a subsequence whose Ces\`{a}ro sequence a.e.-converges to a function in $X$. However, little is known about dominated a.e.-convergence of Ces\`{a}ro sequences. In this paper, we focus on the study of this.

Parallel to the weak Banach-Saks property, one can also introduce ``weak" versions of the order Banach-Saks property. These ``weak" order-type Banach-Saks properties are closely related to (and motivated by) the problem of order closedness of convex sets, where a central topic is to identify a dominated a.e.-convergent sequence of convex combinations of a given sequence that is convergent with respect to certain weak topologies. This problem has direct applications in Mathematical Finance and has lately been systematically investigated by the authors and others; see, e.g., \cite{DO,GLMX,GLX2,GX,TL1,TL2}.

We state the precise definitions of the order-type Banach-Saks properties below. For a Banach function space $X$, denote by $X'$ its associate space (see definition in Subsection 1.1 below).  

\begin{definition}\label{BSP-def1}
	A Banach function space $X$ is said to have
	\begin{enumerate}[$(1)$]
		\item the order Banach-Saks property ($oBSP$) if every norm bounded sequence in $X$ has a subsequence whose Ces\`{a}ro sequence order converges in $X$ (the definition of order convergence is given in Subsection 1.1 below). 
		\item the weak order Banach-Saks property ($w$-$oBSP$) if every weakly null  sequence in $X$ has a subsequence whose  Ces\`{a}ro sequence order converges in $X$.
		\item the $X'$-order Banach-Saks property ($X'$-$oBSP$) if every $\sigma(X,X')$-null norm bounded sequence in $X$ has a subsequence whose Ces\`{a}ro sequence order converges in $X$.
	\end{enumerate}
\end{definition}
\noindent 
It is apparent that $oBSP\Longrightarrow X'\text{-}oBSP\Longrightarrow w\text{-}oBSP$. The $X'\text{-}oBSP$ may seem extraneous at first sight, but Theorem \ref{uo-un-oBSP}  tells us that it is in fact highly relevant. 

\begin{remark}\begin{enumerate}
		\item {Note that any order convergent sequence is norm convergent whenever $X$ is order continuous. Therefore, in case $X$ is order continuous, $oBSP$ (resp. $w$-$oBSP$) will imply $BSP$ (resp. $w$-$BSP$).}
		\item {Recall also that a norm convergent sequence admits order convergent subsequences. However, this fact does not directly conclude that $BSP$ (resp., $w$-$BSP$) implies $oBSP$ (resp., $w$-$oBSP$) because this fact only yields an order convergent \emph{subsequence} of a Cesaro sequence while $oBSP$ (resp., $w$-$oBSP$) requires order convergence of a whole Cesaro sequence. One observes from Proposition \ref{sigma-nec} and in Remark \ref{remark-bsp-obsp} that $L^1$ and any non-reflexive order continuous Orlicz space fails $w$-$oBSP$ although it is known to have $w$-$BSP$. Suprisingly, any Orlicz space with $BSP$ must also have $oBSP$ (Remark \ref{remark-bsp-obsp}).}
	\end{enumerate}
\end{remark}

We organize the paper as follows.  Section 2 contains general results concerning the three types of order Banach-Saks properties in Definition \ref{BSP-def1}. Theorem \ref{uo-un-oBSP} shows that $oBSP$ is closely related to $X'$-$oBSP$. Theorem \ref{lemma-sigma-oBSP} gives a characterization of $w$- and $X'$-$oBSP$ when $X$ admits a special modular in the sense of \cite{TL1}. The result is fine tuned in Theorem \ref{modular-sigma-oBSP} when $X$ is rearrangement invariant. Theorem \ref{modular-sigma-oBSP} is utilized in Section 3 to provide a complete characterization of all three types of order Banach-Saks properties in Orlicz spaces (Theorems \ref{Orlicz-oBSP} and \ref{H-abssigma-oBSP}). The final Section 4 is concerned with hereditary order-type Banach-Saks properties. Using Ramsey type arguments, it is shown that $oBSP$ always implies its hereditary version (Theorem \ref{hoBSP}). The same is true for $w$- and $X'$-$oBSP$ for a large class of Banach function spaces (Theorem \ref{w-hoBSP}).

\subsection{Some definitions and basic facts}
We give basic notions for Banach function spaces, which are a subclass of Banach lattices. The reader is referred to \cite{AB} for the counterparts on Banach lattices, if needed. In this paper, we restrict our attention to function spaces defined on [0,1]. Extension to general nonatomic probability measures does not entail essentially new ideas.  Denote by $L^0$ $(:=L^0[0,1])$ the vector lattice of all (equivalence classes with respect to a.e.\ equality  of) real measurable functions on $[0,1]$ (endowed $[0,1]$ with the Lebesgue measure). A {\em Banach function space}  $X$ on $[0,1]$ is a Banach space of functions in $L^0$  such that $f\in X$ and $\|f\|_{X}\leq \|g\|_{X}$ whenever $g\in X$ and $|f|\leq |g|$. We denote by $X_+$ the positive cone of $X$.  A set $E$ in $X$ is said to be 
{\em dominated} in $X$ if there exists $h\in X_+$ such that $|f|\leq h$ for every $f\in E$. When $E=\{f_n\}$, this is equivalent to that the {\em maximal function} $\sup_n|f_n|$ belongs to $X$.
In the Banach lattice language, a sequence $\{f_n\}$ in $X$ is said to {\em order converge} to $f\in X$, written as $f_n\xrightarrow{o}f$, if there exists $\{g_n\}$ in $X$ such that $g_n\downarrow0$ and $|f_n-f|\leq g_n$ for all $n\in\mathbb{N}$. This is easily seen to be equivalent to that    $\{f_n\}$ is dominated in $X$ and a.e.-converges to $f$. We interchangeably use the terms of dominated a.e.-convergence and order convergence. The latter explains  the name of {\em order} Banach-Saks property. 

Given a Banach function space $X$, its {\em associate space} $X'$   is the space of all $g\in L^0$ such that $\int_0^1|fg|\mathrm{d}t<\infty$ for every $f\in X$. In the Banach lattice literature, it coincides with the order continuous dual $X_n^\sim$ of $X$, i.e., the collection of all linear functionals $\varphi$ on $X$ such that $\varphi(f_n)\to0$ whenever $f_n\xrightarrow{o}0$ (see, e.g., \cite[Theorem 2.6.4]{MN}).  $X'$ can be identified a Banach function space with the dual norm (\cite[Section 2.6]{MN}).
The {\em order continuous part} $X_a$ of $X$ is the set of all  functions in $X$ that have absolutely continuous norm, i.e., all $f\in X$ such that $\|f_n\|_X\to0$ whenever $f_n\xrightarrow{o}0$ and $|f_n|\leq |f|$.
The following two classes of Banach function spaces are important in the paper. A Banach function space $X$ is {\em order continuous} if $X_a=X$, or equivalently, iff $X'=X^*$ (\cite[Theorem 2.4.2]{MN}), iff $X$ is separable (\cite[Corollary 4.52]{AB}).
$X$ is  {\em monotonically complete} if the pointwise supremum of any increasing, norm bounded sequence in  $X_+$ belongs to $X$, or equivalently, if $X$ is isomorphic to $(X')'$ via the evaluation mapping (\cite[Theorem 2.4.22]{MN}) \footnote{In the literature, monotonic completeness bears many other names such as {\em maximal} and {\em the weak Fatou property}. It is sometimes included in the definition of function spaces; the inconvenience is that it excludes natural spaces like the Orlicz hearts as Banach function spaces.}.

For any $h\in L^0$, we denote by $h^*:(0,1]\to[0,\infty)$ the decreasing rearrangement of $h$, that is,
\[
h^*(t)=\inf\Big\{\lambda:\big|\{|h|>\lambda\}\big|\leq t\Big\},\ t\in(0,1].
\]
Here we use $|A|$ to denote the Lebesgue measure of a set $A$.
A rearrangement invariant (r.i.) space is a  Banach function space $X\neq\{0\}$ such that for every $h\in L^0$ satisfying   $h^*=f^*$ for some $f\in X$, $h\in X$ and $\|h\|_{{X}}=\|f\|_{{X}}$. It is well known that $L^\infty\subseteq X\subseteq L^1$ and either $X=L^\infty$ or $L^\infty\subseteq X_a$ (\cite[Lemma 2.2]{TL2}).

A function $\varphi:[0,\infty)\to[0,\infty)$ is called an Orlicz function if it is convex, increasing, non-constant, and $\varphi(0)=0$. Define the conjugate $\psi$ of $\varphi$ by
\[
\psi(s)=\sup\{st-\varphi(t):t\geq0\}
\]
for all $s\geq0$. Note that $\psi$ is an Orlicz function whenever $\lim_{t\to\infty}\frac{\varphi(t)}{t}=\infty$, or equivalently, $\psi$ is finite-valued.  The Orlicz space $L^\varphi$ $(:=L^\varphi[0,1])$ is the space of all $f\in L^0$ such that
\[
\|f\|_{\varphi}:=\inf\left\{\lambda>0:\int_0^1 \varphi\left(\frac{|f|}{\lambda}\right)\mathrm{d}t\leq 1 \right\}<\infty.
\]
$L^\varphi$ with the Luxemburg norm $\|\cdot\|_\varphi$ is an r.i. space. Furthermore, $(L^\varphi)'=L^{\psi}$ if $\lim_{t\to\infty}\frac{\varphi(t)}{t}=\infty$ and $(L^\varphi)'=L^\infty$ otherwise. The order continuous part of $L^\varphi$ is given by the space of all $f\in L^\varphi$ such that
\[
\int_0^1 \varphi\left(\frac{|f|}{\lambda}\right)\mathrm{d}t<\infty
\]
for all $\lambda>0$. This space is also called the Orlicz heart of $L^\varphi$ and is denoted by $H^\varphi$. It is known that $L^\varphi=H^\varphi$ if and only if $L^\varphi$ is order continuous if and only if $\varphi$ has the $\Delta_2$-condition, i.e., there exist $C>0$ and $t_0>0$ such that
\[
\varphi(2t)\leq C\varphi(t),\quad \forall \ t>t_0.
\]
{Recall also that $L^\varphi$ is reflexive if and only if both $\varphi$ and $\psi$ have the $\Delta_2$-condition.} For detailed information regarding Orlicz spaces, we refer to \cite{Ch, LT, Mal}.

\section{Order Banach-Saks properties -- general results}

The following proposition establishes a necessary condition for $oBSP$.

\begin{proposition}\label{obs-smc}
	Let $X$ be a Banach function space.	If $X$ has $oBSP$, then it is monotonically complete and $X^*$ is order continuous.
\end{proposition}
\begin{proof}
	First, we show that $X$ is monotonically complete. Take any norm bounded increasing sequence $\{f_n\}$ in $X_+$. Since $X$ has $oBSP$, by passing to a subsequence we may assume that $\frac{1}{n}\sum_{k=1}^nf_k\stackrel{o}{\rightarrow} f $ for some $f\in X_+$. Take any $m\in\mathbb{N}$. For any $n\geq m$, we have $$\frac{1}{n}\sum_{k=1}^nf_k\geq \frac{n-m+1}{n} f_m+\frac{m-1}{n}f_1. $$
	Letting $n\rightarrow \infty$ and taking the a.e.-limit, we have
	$f\geq f_m$. Hence, the pointwise supremum $ \sup_n f_n$ lies between $0$ and $f $, and consequently,  lies in $ X$. This proves that  $X$ is monotonically complete.
	
	Next, suppose that $X^*$ is not order continuous. Then there is a  disjoint sequence $\{f_n\}$ in $X_+$ that is order isomorphic to the $\ell^1$-basis (see, e.g., \cite[Theorem 2.4.14]{MN}). Since $X$ has $oBSP$, by passing to a subsequence, we may assume that the  Ces\`{a}ro sequence of $\{f_n\}$ order converges in $X$, and in particular is dominated by some $f\in X$. Using disjointness, we have
	$$\sum_{k=1}^n\frac{f_k}{k}=\sup_{m\leq n}\frac{1}{m}\sum_{k=1}^mf_k\in [0,f].$$
	Thus $$\sum_{k=1}^n\frac{1}{k}\leq C\left\|\sum_{k=1}^n\frac{f_k}{k}\right\|_X\leq C\|f\|_X$$
	for a constant $C$ and any $n\in\mathbb{N}$. This contradiction completes the proof.
\end{proof}

As a consequence, we see that $oBSP$ and $X'$-$oBSP$ are closely related.

\begin{theorem}\label{uo-un-oBSP}Let $X$ be a Banach function space.	
	The following   are equivalent:
	\begin{enumerate}[$(1)$]
		\item $X$ has $oBSP$.
		\item $X$ is monotonically complete, $X^*$ is order continuous and $X$ has $X'$-$oBSP$.
		\item $X$ is monotonically complete, $X'$ is order continuous and $X$ has $X'$-$oBSP$.
	\end{enumerate}
\end{theorem}
\begin{proof}
	$(1)\implies (2)$ follows directly from Proposition \ref{obs-smc}. $(2)\implies(3)$ is obvious, since $X'=X_n^\sim$ is a band of $X^*$ (\cite[Theorem 1.57]{AB}) and thus automatically inherits order continuity from $X^*$. It remains to show $(3)\implies(1)$. Assume that $(3) $ holds. Let $\{f_n\}$ be a norm bounded sequence in $X$. It suffices to show that $\{f_n\}$ has a $\sigma(X,X')$-convergent subsequence. Observe that, being an order continuous Banach function space, $X'$ is separable. Thus the unit ball of $(X')^*$ is sequentially $w^*$-compact. Moreover, since $X$ is monotonically complete, it is order isomorphic to $(X')'$ via the canonical evaluation mapping; since $X'$ is order continuous, $(X')'=(X')^*$. Thus identifying $X$ with $(X')^*$, we get the desired result.  
\end{proof}

{From Theorem \ref{uo-un-oBSP}, one can see that in case $X$ is reflexive, $X'$-$oBSP$  will also imply the Banach-Saks property.}  It is thus of great interest to focus on $X'$-$oBSP$. Clearly, if $X$ is order continuous, then $X'$-$oBSP$ and $w$-$oBSP$ are equivalent. We  show that the same result remains true for many r.i.\ spaces. 
Before proceeding to the proof, we first establish a lemma that reduces order convergence to dominatedness in the definitions of $w$-$oBSP$ and $X'$-$oBSP$, which will also be used in Section 4.

\begin{lemma}\label{o-ob-1}Let $X$ be a Banach function space.	
	The following  are equivalent:
	\begin{enumerate}[$(1)$]
		\item $X$ has $w$-$oBSP$ (resp., $X'$-$oBSP$).
		\item Every weakly null (resp., $\sigma(X,X')$-null norm bounded)  sequence in $X$ has a subsequence whose Ces\`{a}ro sequence order converges to 0.
		\item Every weakly null (resp., $\sigma(X,X')$-null norm bounded)  sequence in $X$ has a subsequence whose Ces\`{a}ro sequence is dominated in $X$.
	\end{enumerate}
\end{lemma}
Note that compared to $(1)$, $(2)$ asserts also that the order limit of the Ces\`{a}ro sequence must be equal to the topological limit.

\begin{proof}We prove the case of $w$-$oBSP$; the other case can be proved along the same lines.
	Clearly, $(2)\implies (1)\implies(3)$. 	Assume that $(3)$ holds and let $\{f_n\}$ be a weakly null sequence in $X$. It is well known that for a Banach function space $X$, there exists $0\leq g\in X'$ such that 
	\begin{align}\label{support}
		\big|\{h\neq 0\}\backslash\{g>0\}\big|=0
	\end{align}  
	for every $h\in X$ (see, e.g., \cite[Corollary 5.27]{AA}). Let $\mathrm{d}\nu=\frac{g}{g+1}\mathrm{d}t$. Then $\nu$ is a finite measure on $[0,1]$  and  for any $h\in X$,
	$$\|h\|_{L^1(\nu)}\leq \int_0^1|h|g\mathrm{d}t\leq \|h\|_X\|g\|_{X'},$$
	where $\|g|_{X'}$ is the dual norm of $g$. Thus being norm bounded in $X$, $\{f_n\}$ is also norm bounded in $L^1(\nu)$. By the classical result of Koml\'{o}s \cite{Ko} and by passing to a subsequence if necessary, we may assume that there exists $f\in L^1(\nu)$ such that the Ces\`{a}ro sequence of any subsequence of $\{f_n\}$ converges  to $f$ $\nu$-a.e. We show that $f=0$. This, together with \eqref{support}, would imply that the Ces\`{a}ro sequence of any subsequence of $\{f_n\}$ converges  to $0$ a.e.\ in the Lebesgue measure. 
	By $(3)$, we choose a subsequence $\{f_{n_k}\}$ of $\{f_n\}$ whose Ces\`{a}ro sequence is dominated in $X$. (So once $f=0$ is established, the Ces\`{a}ro sequence order converges to $0$ in $X$.)  Clearly, the Ces\`{a}ro sequence of $\{f_{n_k}\}$ is also dominated in $L^1(\nu)$ and thus converges to $f$ in $\sigma(L^1(\nu),L^\infty(\nu))$ by Dominated Convergence Theorem. On the other hand, since $\{f_{n_k}\}$ is weakly null in $X$, it is $\sigma(X,X')$-null. For any $g'\in L^\infty(\nu)$, since $|g'\frac{g}{g+1}|\leq ||g'||_{L^\infty(\nu)} g$, $g'\frac{g}{g+1}\in X'$. From these, it  follows  immediately that  $\{f_{n_k}\}$ is $\sigma(L^1(\nu),L^\infty(\nu))$-null; in particular, the Ces\`{a}ro sequence of $\{f_{n_k}\}$ is also $\sigma(L^1(\nu),L^\infty(\nu))$-null. Hence, $f=0$. This proves  $(3)\implies (2)$.
\end{proof}

\textbf{For the rest of this section, we always assume that $X$ is an r.i.\ space on $[0,1]$.}
A functional $\rho:X\to[0,\infty]$ is called a \textit{special modular} on $X$ if it satisfies the following conditions:
\begin{enumerate}[$(M1)$]
	\item If $f\in X_+$ and $\rho(f)<\infty$, then $\rho(f_n)\to0$ whenever $f_n\xrightarrow{o}0$ and $0\leq f_n\leq f$.
	\item If $\{f_n\}$ is a sequence in $X$ such that $\sum_{n=1}^\infty\rho(f_n)<\infty$, then a subsequence of $\{f_n\}$ is dominated in $X$.
	\item $\rho(f)<\infty$ for every $f$ in the closed unit ball of $X$.
\end{enumerate}
Order continuous Banach function spaces and  Orlicz spaces admit   special modulars (\cite{TL1}). The latter fact will be used in the next section.

\begin{theorem}\label{lemma-sigma-oBSP}
	Suppose that $X$ admits a special modular. Then the following statements are equivalent:
	\begin{enumerate}[$(1)$]
		\item $X$ has $X'$-$oBSP$.
		\item $X$ has $w$-$oBSP$.
		\item Every weakly null  sequence in $X_a$ has a subsequence whose  Ces\`{a}ro sequence is dominated in $X$.
	\end{enumerate}
\end{theorem}

\begin{proof}
	We may assume that $X\neq L^\infty$ since $L^\infty$ satisfies $(1)$, $(2)$, and $(3)$. Recall that, in this case, $L^\infty\subset X_a$. $(1)\implies (2)\implies(3)$ is clear. Suppose that $(3)$ holds. 
	Let $\{f_n\}$ be a $\sigma(X,X')$-null, norm bounded sequence in $X$. Without loss of generality, assume that $\{\|f_n\|_X\}$ is bounded by 1. Denote by $\rho$ the special modular on $X$.  Using $(M1)$ and $(M3)$, for every $n\in\mathbb{N}$, we can find a measurable set $A_n\subseteq [0,1]$ such that $\rho(f_n\mathbf{1}_{A_n})<\frac{1}{2^n}$, $|A_n|<\frac{1}{2^n}$, and $f_n\mathbf{1}_{A_n^C}\in L^\infty\subset X_a$ (one can choose $A_n$ to be of the form $\{|f_n|>c\}$ for some $c>0$). Let $y_n=f_n\mathbf{1}_{A_n}$ and $z_n=f_n\mathbf{1}_{A_n^C}$. Clearly, both $\{y_n\}$ and $\{z_n\}$ are norm bounded. Furthermore, since $\big|\bigcup_{i=n}^\infty A_i\big|\to0$, $\{y_n\}$ a.e.-converges to 0. Since $\sum_{n=1}^\infty\rho(y_n)<\infty$, by $(M2)$ and passing to a subsequence, we may assume that $\{y_n\}$ is dominated  in $X$. It follows that $\{y_n\}$ order converges to 0 in $X$, in particular, it is $\sigma(X,X')$-null. As both $\{f_n\}$ and $\{y_n\}$ are $\sigma(X,X')$-null, $\{z_n\}$ is also $\sigma(X,X')$-null in $X$ and hence, weakly null in $X_a$ (note that $X_a$, containing $L^\infty$, is order dense in $X$ and thus every member in $(X_a)^*=(X_a)'$ extends to a member in $X'$ by \cite[Lemma 3.3]{GLX1}). By $(3)$ and passing to a subsequence, we may assume that the Ces\`{a}ro sequence of $\{z_n\}$ is dominated in $X$. Since $\{y_n\}$ is   dominated in $X$, its Ces\`{a}ro sequence is also dominated in $X$. Thus, we conclude that the Ces\`{a}ro sequence of $\{f_n\}$ is dominated in $X$. This proves $(3)\implies (1)$  in view of Lemma \ref{o-ob-1}.
\end{proof}

In the next theorem, we will provide sufficient conditions for $w$-$oBSP$ and $oBSP$ when there is a special modular. We need the following two lemmas. Recall that an r.i.\ space $X$ is {\em $p$-convex} if there exists a constant $C>0$ such that $$\left\|\Big(\sum_{i=1}^n|f_i|^p\Big)^{\frac{1}{p}}\right\|_X\leq C\left(\sum_{i=1}^n\|f_i\|_X^p\right)^{\frac{1}{p}}$$  for any choice of finitely many functions $\{f_i\}$ in $X$.

\begin{lemma}\label{p-cvx-oc}
	If $X$ is $p$-convex for some $p\in(1,\infty)$, then $X^*$ is order continuous.
	
\end{lemma}

\begin{proof}
	Assume that $X^*$ is not order continuous. Then there is a norm bounded disjoint sequence $\{f_n\}$ in $X_+$ that is order isomorphic to the $\ell^1$-basis. By disjointness, 
	$$
	\left|	\sum_{k=1}^n\frac{f_k}{k}\right|=\left(\sum_{k=1}^n\frac{|f_k|^p}{k^p}\right)^{\frac{1}{p}}.$$
	Thus by the property of $\ell^1$-basis and $p$-convexity, $$\sum_{k=1}^n\frac{1}{k}\leq C'\left\|	\sum_{k=1}^n\frac{f_k}{k}\right\|_X=C'\left\|	\left(\sum_{k=1}^n\frac{|f_k|^p}{k^p}\right)^{\frac{1}{p}}\right\|_X\leq C''\Big(\sum_{k=1}^n \frac{1}{k^p}\Big)^{\frac{1}{p}} \sup_{k}\|f_k\|_X\leq C'''$$
	for some constants $C',C'',C'''$ and any $n\geq 1$. This is impossible and  finishes the proof.
\end{proof}

\begin{lemma}\label{martingale-diff}
	If $X$ is $p$-convex for some $1<p<2$, then there is a constant $M>0$ such that for every norm bounded sequence $\{f_n\}$ in $X$,
	\[
	\left\|\sqrt{\sum_{k=1}^n |a_kf_k|^2}\right\|_X \leq M\|\{a_n\}\|_p\sup_n\|f_n\|_X 
	\]
	for every $n\in\mathbb{N}$ and every sequence $\{a_n\}$ in $\mathbb{R}$.
\end{lemma}
\begin{proof}
	The conclusion follows from the $p$-convexity of $X$ and the fact that
	\[
	\left\|\sqrt{\sum_{k=1}^n |a_kf_k|^2}\right\|_X \leq\left\|\sqrt[p]{\sum_{k=1}^n |a_kf_k|^p}\right\|_X.
	\]
\end{proof}

For every $s>0$, let $D_s$ be the linear operator on $X$ such that $\left(D_s(f)\right)(t)={f(t\slash s)}$ when $t\leq \min\{1,s\}$ and zero otherwise. Define the (upper) Boyd index $q_X$ of $X$ by
\[
q_X=\inf_{0<s<1}\frac{\log s}{\log \|D_s\|},
\]
where $\|D_s\|$ is the norm of $D_s$ acting as an operator on $X$ (see \cite{LT} for more information about the Boyd index).

\begin{theorem}\label{modular-sigma-oBSP}
	Suppose that $X$ is $p$-convex $(1<p<\infty)$ with $q_X<\infty$.  The following statements hold.
	\begin{enumerate}[$(1)$]
		\item If $X$ is order continuous, then $X $ has $w$-$oBSP$.
		\item If $X$ is   monotonically complete and admits a special modular, then $X$ has $oBSP$.
	\end{enumerate}
\end{theorem}

\begin{proof}
	$(1) $ Assume that $X$ is order continuous. 
	By  Lemma \ref{o-ob-1}, it suffices to show that every weakly null   sequence in $X$ has a subsequence whose Ces\`{a}ro sequence is dominated in $X$. Note that $X$ is $r$-convex for $1< r< \min\{p,2\}$ (\cite[Proposition 1.3]{CT}). Hence, we may assume that $1<p<2$. 
	
	Let $\{f_n\}$ be a weakly null sequence in $X$. Since $X$ is order continuous, it is separable. \cite[Proposition 2.1]{AKS} asserts that for a weakly null sequence $\{f_n\}$ in a separable r.i.\ space $X$, there is a subsequence $\{f_{n_k}\}$ of $\{f_n\}$ such that
	\[
	f_{n_k}=y_k+z_k,
	\]
	where $\{y_k\}$ is a weakly null martingale difference sequence in $X$ (with respect
	to an increasing sequence of $\sigma$-subalgebras of the $\sigma$-algebra of all Lebesgue measurable subsets of $[0,1]$) and $\|z_k\|_X \leq\frac{1}{2^k}$ for each $k\in\mathbb{N}$. Clearly, the norm sum $\sum_{k=1}^\infty |z_k|$ lies in $X$. Thus $\{z_k\}$, and therefore its Ces\`{a}ro sequence, is dominated in $X$. It only remains to show that the Ces\`{a}ro sequence of $\{y_k\}$ is dominated in $X$. 
	
	Without loss of generality, we may assume that $\{\|y_k\|_X\}$ is bounded by 1. Note that $\{\sum_{k=1}^nc_ky_k\}$ is a martingale for any sequence $\{c_n\}$ in $\mathbb{R}$. For each $n\in\mathbb{N}$, define
	\[
	g_n=\sup_{1\leq k\leq n} \left|\sum_{i=1}^k\frac{y_i}{i}\right|.
	\]
	Observe that for $n>m$,
	\[
	0\leq g_n-g_m\leq \sup_{m< k\leq n}\left|\sum_{i=m+1}^k\frac{y_i}{i}\right|.
	\]
	Since $q_X<\infty$, \cite[Theorem 3]{JS}  yields a constant $C$, independent of $m$ and $n$, such that
	\[
	\left\|g_n-g_m\right\|_X\leq \left\|\sup_{m< k\leq n}\left|\sum_{i=m+1}^k\frac{y_i}{i}\right|\right\|_X\leq C\left\|\sqrt{\sum_{k=m+1}^n\left|\frac{y_k}{k}\right|^2}\right\|_X \leq CM\left(\sum_{k=m+1}^n\frac{1}{k^p}\right)^{\frac{1}{p}};
	\]
	the last inequality is due to  Lemma \ref{martingale-diff}. This implies that $\{g_n\}$ is a Cauchy sequence in $X$. One sees that since $\{g_n\}$ is increasing, its norm limit $g$ is also its pointwise supremum (see, e.g., \cite[Theorem 3.46]{AB}). 
	
	Now, observe that
	\[
	\frac{y_1+\dotsb+y_m}{m}=\sum_{k=1}^m\frac{y_k}{k}-h_1-\dotsb-h_{m-1},
	\]
	where $h_k=\frac{1}{k(k+1)}\sum_{i=1}^ky_i$. Applying \cite[Theorem 3]{JS} together with Lemma \ref{martingale-diff} again to {the martingale  $\{\sum_{i=1}^j y_i\}_{j=1}^k$  gives us
		\[
		\|h_k\|_X=\frac{1}{k(k+1)}\left\|\sum_{i=1}^ky_i\right\|_X\leq \frac{1}{k(k+1)}\left\|\sup_{1\leq j\leq k}\left|\sum_{i=1}^jy_i\right|\right\|_X\leq \frac{C}{k(k+1)}\left\|\sqrt{\sum_{i=1}^k|y_i|^2}\right\|_X\leq CM\frac{\sqrt[p]{k}}{k(k+1)}.
		\]
	}It follows that
	\[
	\sum_{k=1}^\infty \|h_k\|_X\leq CM\sum_{k=1}^\infty \frac{\sqrt[p]{k}}{k(k+1)}<\infty.
	\]
	Hence, $\sum_{k=1}^\infty |h_k|\in X$ (both norm and pointwise sum). Since
	\[
	\left|\frac{y_1+\dotsb+y_m}{m}\right|\leq g+\sum_{k=1}^\infty |h_k|
	\]
	for each $m\geq 1$, we obtain that the Ces\`{a}ro sequence of $\{y_k\}$ is dominated in $X$.
	This concludes the proof of the case.
	
	$(2)$	Now, assume that $X$ is  monotonically complete. Note that $X_a$ is order continuous and $p$-convex with $q_{X_a}<\infty$. Thus by the preceding case, $X_a$ has $w$-$oBSP$. It follows from Theorem \ref{lemma-sigma-oBSP} that $X$ has $X'$-$oBSP$.  Recall from Lemma \ref{p-cvx-oc} that  $X^*$ is order continuous. Thus
	Theorem \ref{uo-un-oBSP} implies that $X$ has $oBSP$.
\end{proof}

\section{Characterizations in Orlicz spaces}

In this section, we analyze the order-type Banach-Saks properties for Orlicz spaces and Orlicz Hearts  on $[0,1]$. 
Let $\varphi$ be an Orlicz function so that its conjugate $\psi$ is also an Orlicz function. Recall that $\left(L^\varphi\right)'=L^{\psi}$.   We say that an Orlicz function $\varphi$ has {\em property $(H)$} if every weakly null sequence in $H^\varphi$ has a subsequence whose Ces\`{a}ro sequence is dominated in $L^\varphi$. The main result in this section is the following characterization (and equivalence) of all three order-type  Banach-Saks properties  in Orlicz spaces.

\begin{theorem}\label{Orlicz-oBSP}
	The following statements are equivalent:
	\begin{enumerate}[$(1)$]
		\item $L^\varphi$ has $oBSP$.
		\item $L^\varphi$ has $L^\psi$-$oBSP$.
		\item $L^\varphi$ has $w$-$oBSP$.
		\item $\varphi$ has property $(H)$.
	\end{enumerate}
\end{theorem}

The equivalence of $(2)$, $(3)$ and $(4)$ is simply Theorem \ref{lemma-sigma-oBSP}. $(1)\implies(2)$ is obvious. We proceed to the proof of  $(2)\implies(1)$.
We first establish the following auxiliary result. 

\begin{proposition}\label{sigma-nec}
	Let $X=L^1$ or $X=L^\varphi$ such that $\psi$ fails the $\Delta_2$-condition. Then $X$ fails $X'$-$oBSP$.
\end{proposition}
\begin{proof}
	We will only provide the proof for the case $X=L^\varphi$ whose $\psi$ fails the $\Delta_2$-condition. The case $X=L^1$ can be proved similarly.
	
	Suppose that $\psi$ fails the $\Delta_2$-condition but $L^\varphi$ has $L^\psi$-$oBSP$. For every $n\in\mathbb{N}$, let
	\[
	C_n=\sqrt{\sum_{m=1}^n\frac{1}{m}}\quad \text{and} \quad S_n=\frac{1}{C_n}.
	\]
	Recall that $\psi$ fails the $\Delta_2$-condition if and only if for every $L>1$ and $t_0\geq0$, there exists $t>t_0$ such that
	\[
	\varphi(t)>\frac{1}{2L}\varphi(Lt)
	\]
	(\cite[Theorem 4.2]{KR}). Using this property, we can find a strictly increasing sequence $\{a_n\}$ that satisfies the following conditions:
	\begin{enumerate}[$(a)$]
		\item $a_n\uparrow\infty$,
		\item for every $n\in\mathbb{N}$, \begin{equation}\label{LPhi1}
			\varphi\left(\frac{a_n}{nC_n}\right)\geq \frac{1}{2nC_n} \varphi\left(a_n\right).
		\end{equation}
		\item the sequence $\{d_n\}$, defined recursively by $d_0=\frac{1}{4}$ and $d_{n}-d_{n-1}=\frac{S_{n}-S_{n+1}}{\varphi(a_{n})}$, satisfies
		\[
		d_{n+1}-d_{n}\leq \frac{d_{n}}{2^{n}}\left(1-\frac{1}{2^{\frac{1}{n}}}\right)
		\]
		for every $n\geq 1$.
	\end{enumerate}
	Note that the sequence $\{d_n\}$ is strictly increasing and satisfies $d_{n+1}\leq\left(1+\frac{1}{2^n}\right)d_n$ for every $n\geq 0$. Since the sequence $\left\{\prod_{i=0}^n \left(1+\frac{1}{2^i}\right)\right\}_n$ is bounded in $\mathbb{R}$, we obtain that $d_n\uparrow d$ for some $d\in(0,\infty)$ and hence,
	\[
	d-d_n=\sum_{i=n+1}^\infty (d_i-d_{i-1})\leq \left(1-\frac{1}{2^{\frac{1}{n}}}\right) d\left(\frac{1}{2^n}+\frac{1}{2^{n+1}}+\dotsb\right)\leq \left(1-\frac{1}{2^{\frac{1}{n}}}\right) d.
	\]
	This implies that for every $n\geq 0$,
	\begin{equation}\label{LPhi2}
		\left(\frac{d_n}{d}\right)^n\geq \frac{1}{2}.
	\end{equation}
	Let $b_n:=\frac{d_n}{d}$. Then {$\{b_n\}$ is a sequence in $[0,1]$} with $b_n\uparrow1$ and $\varphi(a_{n})(b_{n}-b_{n-1})=\frac{S_{n}-S_{n+1}}{d}$ for each $n\geq1$. 
	
	Define
	\[
	f=\sum_{i=1}^\infty a_i\mathbf{1}_{[b_{i-1},b_i)}.
	\]
	{Since
		\[
		\int_0^1 \varphi(f)\mathrm{d}s=\sum_{i=1}^\infty \varphi(a_i)(b_i-b_{i-1})=\sum_{i=1}^\infty \frac{S_{i}-S_{i+1}}{d}=\frac{S_1-\displaystyle{\lim_{n\to\infty} S_n}}{d}=\frac{1}{d},
		\]
		we have
		\[
		\int_0^1 \varphi\left(\frac{|f|}{\lambda}\right)\mathrm{d}s\leq 1,
		\]
		where $\lambda=\max\{\frac{1}{d},1\}$. Hence, $f\in L^\varphi_+$.}
	
	{Note that the Lebesgue measure on $[0,1]$ is non-atomic. By \cite[Proposition A.31]{FS}, there exists an independent and identically distributed sequence of random variables $\{f_n\}$ with the same distribution as $f$. Furthermore, since $f\in L^\varphi_+$, each $f_n$ lies in $L^\varphi_+$.}
	
	{Since $\{f_n\}$ has the same distribution as $f$, it is uniformly integrable in $L^1$.} Therefore $\{f_n\}$ is relatively weakly compact in $L^1$ and hence, by passing to a subsequence, we may assume $\{f_{n}\}$ converges weakly to some $h$ in $L^1$. On the other hand, $\{f_n\}$ is relatively compact for $\sigma(L^\varphi,L^{\psi})$ by \cite[Corollary 42]{Fr}. Then we deduce that $h\in L^\varphi$ and $\{f_{n}\}$ $\sigma(L^\varphi,L^{\psi})$-converges to $h$. Since $L^\varphi$ has $L^\psi$-$oBSP$, by passing to a subsequence once again, we may assume that the Ces\`{a}ro sequence of $\{f_n\}$ is dominated in $L^\varphi$. We will show that this is impossible. In particular, we will show that $\{h_n\}$ with $$h_n:=\sup_{1\leq m\leq n} \frac{f_1+\dotsb+f_m}{m}$$ is unbounded in norm in $L^\varphi$.
	
	Fix $n\in\mathbb{N}$. Define $f_1',\dotsc,f_n'$ as functions on $[0,1]^n$ such that
	\[
	f_k'(x_1,\dotsc,x_n)=f(x_k)
	\]
	for $k=1,\dotsc,n$. Let $h_n':=\sup_{1\leq m\leq n} \frac{f_1'+\dotsb+f_m'}{m}$. Since $(f_1,\dots,f_n)$ and $(f_1',\dots,f_n')$ have the same joint distribution, $h_n$ and $h_n'$ have the same distribution.   Observe that
	\[
	h_n'(x_1,\dotsc,x_n)\geq \frac{f_m'(x_1,\dotsc,x_n)}{m}=\frac{f(x_m)}{m}
	\]
	for $m=1,\dotsc,n$. Let $A_m:=\left\{(x_1,\dotsc,x_n):x_m>x_i\ \forall i\neq m\right\}$. Then
	\begin{eqnarray*}
		\int_0^1 \varphi\left(\frac{h_n}{C_n}\right)\mathrm{d}s&=&	\int_{[0,1]^n} \varphi\left(\frac{h_n'}{C_n}\right)dx_1\dotsc dx_n\geq\sum_{m=1}^n\int_{A_m} \varphi\left(\frac{f(x_m)}{mC_n}\right)dx_1\dotsc dx_n\\
		&=&\sum_{m=1}^n\int_{0}^1 y^{n-1}\varphi\left(\frac{f(y)}{mC_n}\right)dy\\
		&=&\frac{1}{n}\sum_{m=1}^n\sum_{i=1}^\infty \varphi\left(\frac{a_i}{mC_n}\right)(b_{i}^n-b_{i-1}^n)\\
		&\geq&\sum_{m=1}^n\sum_{i=1}^\infty \varphi\left(\frac{a_i}{mC_n}\right)(b_{i}-b_{i-1})b_{i-1}^{n-1}\\
		&\geq&\sum_{m=1}^n\sum_{i=n}^\infty \varphi\left(\frac{a_i}{mC_n}\right)(b_{i}-b_{i-1})b_{n-1}^{n-1}\\
		&\geq&\frac{1}{2}\sum_{m=1}^n\sum_{i=n}^\infty \varphi\left(\frac{a_i}{mC_n}\right)(b_{i}-b_{i-1})\qquad (\text{by}\ (\ref{LPhi2})).
	\end{eqnarray*}
	Observe that the convexity of $\varphi$ and (\ref{LPhi1}) give us
	\[
	\varphi\left(\frac{a_i}{mC_n}\right)\geq \frac{iC_i}{mC_n}\varphi\left(\frac{a_i}{iC_i}\right)\geq\frac{1}{2mC_n} \varphi\left(a_i\right)
	\]
	for every $i\geq n\geq m$. It follows that
	\[
	\int_0^1 \varphi\left(\frac{h_n}{C_n}\right)\mathrm{d}s\geq\frac{1}{4C_n}\sum_{m=1}^n\frac{1}{m}\sum_{i=n}^\infty \varphi\left(a_i\right)(b_{i}-b_{i-1})=\frac{1}{4C_n}\left(\sum_{m=1}^n\frac{1}{m}\right)\frac{S_n}{d}=\frac{1}{4d}.
	\]
	Since $4d>4d_0=1$, using the convexity of $\varphi$ once again we have
	\[
	\int_0^1 \varphi\left(\frac{4dh_n}{C_n}\right)\mathrm{d}s\geq4d\int_0^1 \varphi\left(\frac{h_n}{C_n}\right)\mathrm{d}s\geq1,
	\]
	or equivalently, $\|h_n\|_{\varphi}\geq \frac{C_n}{4d}$. We deduce that $\{h_n\}$ is unbounded in norm. Thus, $L^\varphi$ fails $L^\psi$-$oBSP$.
\end{proof}

Note that from Proposition \ref{sigma-nec}, one can see that $L^1$  fails the weak order Banach-Saks property although it has the weak Banach-Saks property.

Now, we are ready to present the proof of  Theorem \ref{Orlicz-oBSP}.
\begin{proof}[Proof of Theorem \ref{Orlicz-oBSP}]Suppose that $(2)$ holds. By Proposition \ref{sigma-nec},
	$\psi$ has the $\Delta_2$-condition, or equivalently, $L^{\psi}$ is order continuous. Note that Orlicz spaces are monotonically complete. By Theorem \ref{uo-un-oBSP}, we deduce that $(2)\implies(1)$.
\end{proof}

Note that when $\varphi$ and $\psi$ have the $\Delta_2$-condition (i.e., $L^\varphi$ is reflexive ), $L^\varphi$ is $p$-convex and its Boyd index is finite (in fact, the converse is also true; see \cite{LT}). Since any Orlicz space admits a special modular, Theorem \ref{modular-sigma-oBSP} gives us the following sufficient condition for $L^\varphi$ to have $oBSP$.

\begin{proposition}\label{Orlicz-weak-oBSP}
	If both $\varphi$ and $\psi$ have the $\Delta_2$-condition, then $L^\varphi$ has $oBSP$. 
\end{proposition}

The following is immediate by Proposition \ref{Orlicz-weak-oBSP} and Proposition \ref{sigma-nec}.

\begin{corollary}\label{Orlicz-oBSP-D2}
	Suppose that $\varphi$ has the $\Delta_2$-condition. Then $\varphi$ has property $(H)$ if and only if $\psi$ has the $\Delta_2$-condition.
\end{corollary}

{We summarize the relations between $(H)$  and the $\Delta_2$-conditions as follows:
	\begin{align*}
		\ &\psi\text{ has the }\Delta_2\text{-condition}\\
		\leq\  &\varphi\text{ has } (H)\\
		\leq\  &\varphi\text{ and }\psi\text{ have the }\Delta_2\text{-condition}.
	\end{align*}
	Can we find   non-trivial Orlicz functions $\varphi$ (i.e., $L^\varphi\neq L^1$ and $L^{\psi}\neq L^1$) to make the two relations strict? Whether the last relation is strict is equivalent to whether the converse of Proposition \ref{Orlicz-weak-oBSP} holds and to whether $oBSP$ of $L^\varphi$ implies $\varphi$ has the $\Delta_2$-condition. }

\begin{remark}\label{remark-bsp-obsp}
	
	\begin{enumerate}\item {Recall from \cite[Theorem 5.5]{DSS} that $L^\varphi$ has $w$-$BSP$ if and only if $\varphi$ has the $\Delta_2$-condition. Hence, whenever $\varphi$ has the $\Delta_2$-condition but $\psi$ does not, $L^\varphi$ has $w$-$BSP$ but fails $w$-$oBSP$. }
		\item { It is well known that $L^\varphi$ has $BSP$ if and only if $\varphi$ and $\psi$ both have the $\Delta_2$-condition. Hence, if $L^\varphi$ has $BSP$, then it has $oBSP$. We are not sure whether a general  Banach function space with $BSP$ must satisfy $oBSP$.}
	\end{enumerate}
\end{remark}

As another consequence, we also have the following concrete characterization of the order-type Banach-Saks properties in Orlicz hearts $H^\varphi$. Note that for $H^\varphi$, $w$-$oBSP$ and $X'$-$oBSP$ are equivalent.

\begin{theorem}\label{H-abssigma-oBSP}
	The following statements are equivalent:
	\begin{enumerate}[$(1)$]
		\item $H^\varphi$ has $oBSP$.
		\item $H^\varphi$ has $w$-$oBSP$.
		\item {$H^\varphi$ has $BSP$.}
		\item $\varphi$ and $\psi$ have the $\Delta_2$-condition.
	\end{enumerate}
\end{theorem}

\begin{proof}
	$(1)\implies(2)$ is obvious. {$(1)\implies (3)$ is also clear since $H^\varphi$ is order continuous. Since the Banach Saks property implies reflexivity, which is equivalent to $(4)$ for $H^\varphi$, we have $(3)\implies (4)$. Now, if $(4)$ holds, then $H^\varphi=L^\varphi$ has $oBSP$ by Proposition \ref{Orlicz-weak-oBSP}; therefore, $(4)\implies(1)$. }
	
	{It remains to show $(2)\implies (4)$. Assume that $(2) $ holds. Then $H^\varphi$ has $w$-$BSP$ (by the order continuity of $H^\varphi$). By \cite[Theorem 5.5]{DSS}, $\varphi$ has the $\Delta_2$-condition and hence, $H^\varphi=L^\varphi$.
		$(4) $ now follows from Theorem \ref{Orlicz-oBSP}  and Corollary \ref{Orlicz-oBSP-D2}.}
\end{proof}

Recall from \cite[Theorem 5.5]{DSS} that $H^\varphi$ has $w$-$BSP$ if and only if $\varphi$ has the $\Delta_2$-condition.

\section{Hereditary properties}

It is well known that the Banach-Saks property is equivalent to its ``hereditary'' version, that is, every norm bounded sequence in $X$ has a subsequence such that the Ces\`{a}ro sequence of any further subsequence converges in $X$ (see \cite{FiS}). The same result is also true for the weak Banach-Saks property. In this section, we will show that the order Banach-Saks property is equivalent to its hereditary version. Furthermore, the same results remain true for the ``weak'' order-type Banach-Saks properties in a large class of Banach function spaces.  In the rest of the section, let $X$ denote a Banach function space.

\begin{definition}
	Let $X$ be a Banach function space. We say that $X$ has  the hereditary order Banach-Saks property $(hoBSP)$ if every norm bounded sequence in $X$ has a subsequence such that the Ces\`{a}ro sequence of any further subsequence order converges to the same limit in $X$. The other two types of hereditary order-type Banach-Saks properties, $w$-$hoBSP$ and $X'$-$hoBSP$ are defined similarly.	
\end{definition}

\subsection{Application of a Ramsey theorem}
To get relations between order-type Banach-Saks properties and their hereditary versions, we need the following version of Ramsey's Theorem by Galvin and Prikry (\cite{GP}). For any infinite subset $N$ of $\mathbb{N}$, denote by $[N]$ the collection of all infinite subsets of $N$. We endow $[\mathbb{N}]$ with the usual product topology.

\begin{proposition}\label{GP}
	(\cite[Corollary 6]{GP}) Let $B_n, n\in\mathbb{N},$ be Borel subsets of $[\mathbb{N}]$. Then there exists $M\in[\mathbb{N}]$ such that either $[M]\subseteq B_n$ for infinitely many $n\in\mathbb{N}$ or $[M]\cap B_n=\emptyset$ for every $n\in\mathbb{N}$.
\end{proposition}

Using Proposition \ref{GP} and a similar technique as in \cite{FiS}, we obtain two useful lemmas, Lemma \ref{BSP-l2} and Lemma \ref{BSP-l1}.

\begin{lemma}\label{BSP-l2}
	For every sequence $\{f_n\}$ in $X$, there exists a subsequence $\{h_n\}$ of $\{f_n\}$ such that either
	\[
	\lim_{r\to\infty}\sup_{s>r}\left\|\bigvee_{i=1}^s\left|\frac{1}{i}\sum_{k=1}^i h_{n_k}\right|-\bigvee_{i=1}^r\left|\frac{1}{i}\sum_{k=1}^i h_{n_k}\right|\right\|_X=0\ \forall \{n_k\}
	\]
	or
	\[ \lim_{r\to\infty}\sup_{s>r}\left\|\bigvee_{i=1}^s\left|\frac{1}{i}\sum_{k=1}^i h_{n_k}\right|-\bigvee_{i=1}^r\left|\frac{1}{i}\sum_{k=1}^i h_{n_k}\right|\right\|_X>0 \ \forall \{n_k\}.
	\]
\end{lemma}
\begin{proof}
	Let $\{f_n\}$ be a sequence in $X$. For each $s$ and $r$ with $s>r$, define $T_{s,r}:[\mathbb{N}]\to\mathbb{R}$ by
	\[
	T_{s,r}(M)=\left\|\bigvee_{i=1}^s\left|\frac{1}{i}\sum_{k=1}^i f_{n_k}\right|-\bigvee_{i=1}^r\left|\frac{1}{i}\sum_{k=1}^i f_{n_k}\right|\right\|_X,
	\]
	where $\{n_k\}$ is the  elements of $M$ in increasing order. Note that for any $M\in[\mathbb{N}]$, $T_{s,r}$ is constant on a small neighborhood of $M$. Hence, it is continuous on $[\mathbb{N}]$. Therefore, $\lim_{r\to\infty}\sup_{s>r} T_{s,r}$ is a Borel function on $[\mathbb{N}]$ and  the set $$B=\left\{M\in[\mathbb{N}]:\lim_{r\to\infty}\sup_{s>r} T_{s,r}(M)=0\right\}$$ is a Borel set. By Proposition \ref{GP}, there exists a set $M\in[\mathbb{N}]$ such that either $[M]\subseteq B$ or $[M]\cap B=\emptyset$. The conclusion follows.
\end{proof}

In case that $X$ is order continuous, the first condition of Lemma \ref{BSP-l2} is equivalent to the Ces\`{a}ro sequence of any subsequence of $\{h_n\}$ is dominated in $X$:

\begin{lemma}\label{BSP-l3}
	Let $X$ be an order continuous Banach function space. A sequence $\{f_n\}$ in $X$ is dominated in $X$ if and only if
	\[
	\lim_{r\to\infty}\sup_{s>r}\left\|\bigvee_{n=1}^s |f_n|-\bigvee_{n=1}^r |f_n|\right\|_X=0.
	\]
\end{lemma}
\begin{proof}
	Dominatedness of  $\{f_n\}$ is evidently equivalent to that of  $\{\bigvee_{k=1}^n|f_k|\}$. The latter sequence is a nonnegative increasing sequence. Hence, its dominatedness in an order continuous space is equivalent to norm convergence. 
\end{proof}

\begin{lemma}\label{BSP-l1}
	For every sequence $\{f_n\}$ in $X$, there exists a subsequence $\{h_n\}$ of $\{f_n\}$ so that either
	\[
	\sup_{j}\left\|\bigvee_{i=1}^j\left|\frac{1}{i}\sum_{k=1}^i h_{n_k}\right|\right\|_X<\infty\ \forall \{n_k\}\quad\text{or}\quad \sup_{j}\left\|\bigvee_{i=1}^j\left|\frac{1}{i}\sum_{k=1}^i h_{n_k}\right|\right\|_X=\infty\ \forall \{n_k\}.
	\]
\end{lemma}
\begin{proof}
	The proof is similar to the proof of Lemma \ref{BSP-l2}. We only need to  replace $T_{s,r}(M)$  by $T_j(M)=\left\|\bigvee_{i=1}^j\left|\frac{1}{i}\sum_{k=1}^i f_{n_k}\right|\right\|_X$ and consider the Borel sets $B_p=\left\{M\in[\mathbb{N}]:\sup_{j} T_{j}(M)<p\right\}$, $p\in\mathbb{N}$. By Proposition \ref{GP}, we can find $p$ and a subsequence $\{h_n\}$ of $\{f_n\}$ so that either
	\[
	\sup_{j}\left\|\bigvee_{i=1}^j\left|\frac{1}{i}\sum_{k=1}^i h_{n_k}\right|\right\|_X<p\ \forall \{n_k\}\quad\text{or}\quad \sup_{j}\left\|\bigvee_{i=1}^j\left|\frac{1}{i}\sum_{k=1}^i h_{n_k}\right|\right\|_X=\infty\ \forall \{n_k\}.
	\]
\end{proof}

This lemma will be coupled with the following. 

\begin{lemma}\label{BSP-l4}
	Let $X$ be monotonically complete. A sequence $\{f_n\}$ in $X$ is dominated in $X$ if and only if
	\[
	\sup_n\left\|\bigvee_{i=1}^n |f_n|\right\|_X<\infty.
	\]
\end{lemma}

\subsection{Automatic hereditariness}

Using Lemmas \ref{BSP-l1} and  \ref{BSP-l4}, we can show that $oBSP$ and $hoBSP$ are always equivalent. 

\begin{theorem}\label{hoBSP}
	A Banach function space $X$ has $oBSP$ if and only if it has $hoBSP$.		
\end{theorem}
\begin{proof}
	Clearly, $hoBSP$ implies $oBSP$. Now, suppose that $X$ has $oBSP$. Let $\{f_n\}$ be norm bounded in $X$. 
	Recall from Proposition \ref{obs-smc} that $X$ is monotonically complete. Combining Lemma \ref{BSP-l1}, Lemma \ref{BSP-l4} and $oBSP$, we can find a subsequence $\{f_{n_k}\}$ of $\{f_n\}$ such that the Ces\`{a}ro sequence of any subsequence of $\{f_{n_k}\}$ is dominated in $X$. For notational convenience, one may replace $\{f_n\}$ with $\{f_{n_k}\}$. Now, as in the proof of Lemma~\ref{o-ob-1}, by passing to a subsequence,  there exists $f\in L^0$ such that the Ces\`{a}ro sequence of any subsequence of $\{f_{n}\}$ a.e.-converges to $f$.  Due to dominatedness of the Ces\`{a}ro sequences, we have that $f$ is dominated by some function in $X$ and thus $f\in X$. Therefore, the Ces\`{a}ro sequence of any subsequence of $\{f_{n}\}$ order converges to $f$. Thus, $X$ has $hoBSP$.
\end{proof}

Under additional assumptions, $w$-$oBSP$ and $w$-$hoBSP$ (resp., $X'$-$oBSP$ and $X'$-$hoBSP$) are also equivalent.

\begin{theorem}\label{w-hoBSP}
	Suppose that $X$ is either order continuous or monotonically complete. Then $w$-$oBSP$ and $w$-$hoBSP$ (resp., $X'$-$oBSP$ and $X'$-$hoBSP$) are equivalent.
\end{theorem}
\begin{proof}
	If $X$ is monotonically complete, the same arguments as in the proof of Theorem \ref{hoBSP} work. If  $X$ is order continuous, then apply the same arguments but use Lemmas \ref{BSP-l2} and \ref{BSP-l3} instead of \ref{BSP-l1} and \ref{BSP-l4}.
\end{proof}

\bibliographystyle{plain}

\end{document}